\documentclass[journal]{IEEEtran}
\IEEEoverridecommandlockouts                              
\bstctlcite{IEEEexample:BSTcontrol}                                                          
\usepackage{flushend} 
\usepackage{romannum}
\usepackage[usenames,dvipsnames]{xcolor}
\usepackage{calrsfs}
\DeclareMathAlphabet{\pazocal}{OMS}{zplm}{m}{n}
\usepackage{graphicx}
\usepackage{graphics} 
\usepackage{epsfig} 
\usepackage{mathptmx}
\usepackage{caption}
\usepackage{subcaption}
\usepackage{times} 
\usepackage{tikz}
\usepackage{tkz-berge}
\usetikzlibrary{fit,shapes}                                     
\usetikzlibrary{positioning}
\usepackage{amsmath, amssymb}

\usepackage{amsthm}
\usepackage{amsfonts} 
\usepackage{setspace}
\usepackage{multirow}
\usepackage{textcomp}
\usepackage[english]{babel}
\usepackage{float} 
\usepackage{algorithmicx}
\usepackage{algorithm}
\usepackage{algpascal}
\usepackage{algc}
\usepackage{algcompatible}
\usepackage{algpseudocode}
\let\oldReturn\Return
\renewcommand{\Return}{\State\oldReturn}
\usepackage{mathtools}
\usepackage{bm}
\usepackage{mathptmx}
\usepackage{times} 
\usepackage{mathtools, cuted}
\usepackage{textcomp}
\usepackage[english]{babel}
\usepackage{rotating}
\usepackage{pgfplots}
\pgfplotsset{compat=1.5}
\usepackage{siunitx}
\usepackage{pgfplotstable}
\usepackage{filecontents}

\newcommand{\norm[1]}{\left\lVert#1\right\rVert}

\newcommand{\B}{\pazocal{B}}
\newcommand{\C}{\pazocal{C}}
\newcommand{\K}{\pazocal{K}}

\newcommand{\D}{\pazocal{D}}

\newcommand{\X}{\pazocal{X}}

\newcommand{\x}{{X}}
\renewcommand{\u}{{U}}

\newcommand{\remove}[1]{}

\def \mI{{\mathbb{I}}}

\def \cC{{\cal C}}
\def \cCl{{\cC^\ell}}

\def \*{\star}
\def \10n{\!\!\!\!\!\!\!\!\!\!}

\newcommand{\bK}{\bar{K}}

\newcommand{\R}{\mathbb{R}}
\newcommand{\bA}{\bar{A}}
\newcommand{\bB}{\bar{B}}
\newcommand{\bC}{\bar{C}}

\newtheorem{theorem}{Theorem}
\newtheorem{lem}{Lemma}
\newtheorem{defn}{Definition}
\newtheorem{cor}{Corollary}

\newtheorem{prob}{Problem}

\newtheorem{assume}{Assumption}
\newtheorem{prop}{Proposition}

\graphicspath{{Figures/}}

\title{\LARGE \bf Sparsest Feedback Selection for Structurally Cyclic Systems with Dedicated Actuators and Sensors in Linear Time}
\author{Shana~Moothedath,
        Prasanna~Chaporkar
        and~Madhu~N.~Belur
\thanks{The authors are in the Department of Electrical Engineering, Indian Institute of Technology Bombay, India. Email: $\lbrace$shana, chaporkar, belur$\rbrace$@ee.iitb.ac.in.}}

\begin{document}
\maketitle
\thispagestyle{empty}
\pagestyle{empty}

\begin{abstract}
This paper solves the sparsest feedback selection problem for linear time invariant structured systems,
a long-standing open problem in structured systems. 
We consider structurally cyclic systems with dedicated inputs and outputs. 
We prove that finding a
sparsest feedback selection is of linear complexity for the case of 
structurally cyclic systems with dedicated inputs and outputs.
This problem has received attention recently but key errors in the hardness-proofs 
have resulted in an erroneous conclusion there. 
This is also elaborated in this brief paper
together with a counter-example.
\end{abstract}
\begin{IEEEkeywords}
Linear structured systems, Arbitrary pole placement, Linear output feedback, 
Sparsest feedback selection.
\end{IEEEkeywords}
\section{Introduction}\label{sec:intro}
Feedback selection for control systems that guarantees desired closed-loop performance is a fundamental design problem in control theory. The challenging part of the design problem is to accomplish an optimal design, for example in the sense of number of connections or cost of connections. 
We consider feedback selection in large scale linear dynamical systems. 
However, the analysis done in this paper is based on the zero/non-zero (referred as {\it sparsity}) pattern of the system. 
The rationale behind performing this analysis is, in most large scale systems and real time systems, the numerical values of the non-zero entries in the system description are either not known at all, like social networks, biological systems, or they are not known accurately, like electric networks, power grids, robotics. 
To this end, various system properties of these systems are studied using the sparsity pattern of the system referred as {\it structural analysis} \cite{Rei:88}.

Structural analysis of control systems, namely {\it structural controllability} was introduced by Lin in \cite{Lin:74}. 
Research in this area has become relevant due to applicability in various complex systems:
see \cite{LiuBar:16} and references therein for details. This paper discusses sparsest feedback selection that guarantees arbitrary pole placement. Necessary and sufficient graph theoretic condition that guarantee 
arbitrary pole placement is given in \cite{PicSezSil:84} using the concept of 
fixed modes \cite{WanDav:73}.

Given a large scale dynamical system, our aim is to find a minimum set of feedback edges, i.e., which output to be fed to which input, that arbitrary pole placement of the closed-loop system is possible. In other words, given the digraph representing the state dynamics, the inputs and the outputs of the system, our objective is to find a minimum set of feedback connections that ensure the desired design objective. 

Finding sparsest feedback matrix for a given structured system is considered in \cite{Sez:83}. 
The approach proposed requires a minimum input-output set to be found, which in itself is an NP-hard problem \cite{PeqSouPed:15}. 
The authors in \cite{UnySez:89} discuss minimum cost feedback selection, 
which is a more general problem. 
The method proposed there requires solving a multi-commodity network flow problem:
an NP-hard problem. 
Thus neither \cite{Sez:83} nor  \cite{UnySez:89} can yield
a polynomial time algorithm to the feedback selection problem, and hardness of the 
sparsest feedback selection problem remained unsolved.
For a structured state matrix, the problem of finding {\em jointly} sparsest input, output and 
feedback matrices is addressed in \cite{PeqKarAgu_2:16}. 
On the contrary, \cite{PeqKarPap:15} considers the problem when there is {\em no 
flexibility} in choosing the input and output matrices:  
given structured state, input and output matrices and cost associated with each of them 
(i.e, each input, output and feedback edge is associated with cost), find 
the minimum cost input-output set and the feedback matrix.
This problem is known to be NP-hard and hence \cite{PeqKarPap:15} considers a special class 
of systems where the state matrix is 
{\it irreducible}\footnote{A graph 
   is said to be irreducible if there exists a directed path between 
   any two vertices in the graph, i.e., strongly connected.}. 
For the case when the state, input and output matrices are fixed, the problem of finding 
a sparsest {\em feedback matrix} has been formulated 
in \cite{CarPeqAguKarPap:15, CarPeqAguKarPap:15_arx}, 
where the authors claim and `prove' the NP-hardness of the problem.  
Later below in Section~\ref{sec:fallacies} we elaborate about how NP-hardness
of finding a {\em particular} solution with special properties
(namely, when the sparsest solution's closed-loop
system digraph has exactly two SCCs) is not sufficient for a reduction
procedure in general. 
A counter-example is also described there.

In the context of NP-hardness of the 
sparsest feedback selection problem, we proved the NP-hardness recently
in \cite{MooChaBel:17_arxMinCostFeedback} 
using a reduction of the set cover problem.
In this paper we formulate
a subclass of systems referred as 
structurally cyclic systems with dedicated inputs and outputs: for this
subclass, we prove that finding a sparsest feedback matrix has linear
time complexity.  A system is said 
to be structurally {\em cyclic} if the state bipartite graph 
(see Section~\ref{sec:prelim}) has a perfect matching and an input (output, resp.) 
is said to be {\em dedicated} if it can actuate (sense, resp.) a single state only. 
The class of systems with the state bipartite graph having a perfect matching is wide:
for example, {\it self-damped} systems (see \cite{ChaMes:13}) including consensus dynamics 
in multi-agent systems and epidemic equations. 
Further, for systems whose system dynamics are invertible, the 
state bipartite graph necessarily has a perfect matching. 


 The paper is organized as follows.
 In Section~\ref{sec:prob} we formally define the problem that is considered here. 
Here, we also provide some required preliminaries and state some known results that 
we use subsequently. 
In Section~\ref{sec:fallacies}, we elaborate on the fallacies in the proof of NP-hardness result of the proposed problem given in \cite{CarPeqAguKarPap:15}. 
In Section~\ref{sec:poly}, we provide a linear complexity algorithm for solving the proposed problem. Finally, we conclude in Section~\ref{sec:conclu}.

\section{Problem Formulation and Preliminaries}\label{sec:prob}
In this section we formulate the problem considered in this paper 
and then give few preliminaries used in the sequel. 
\subsection{Problem Formulation}\label{sec:formulation}
Consider a linear time-invariant system  $\dot{x} = Ax+Bu$, $y=Cx$, where $A \in \R^{n \times n}$, $B \in \R^{n \times m}$ and $C \in \R^{p \times n}$. Here $\R$ denotes the set of real numbers. The structural representation of this system referred as {\it structured system} is denoted by $(\bA, \bB, \bC)$, where $\bA, \bB$ and $\bC$ has the same structure as that of $A, B$ and $C$ respectively. More precisely,
\begin{eqnarray}\label{eq:struc}
A_{ij} &=& 0 \mbox{~whenever~} \bA_{ij} = 0,\mbox{~and} \nonumber \\
B_{ij} &=& 0 \mbox{~whenever~} \bB_{ij} = 0,\mbox{~and} \nonumber \\
C_{ij} &=& 0 \mbox{~whenever~} \bC_{ij} = 0.
\end{eqnarray}
Given $(\bA, \bB, \bC)$, any tuple $(A, B, C)$ that satisfies \eqref{eq:struc} is referred as a {\it numerical realization} of the structured system. Let $\bK \in \{0, \* \}^{m \times p}$ denote a feedback matrix, where $\bK_{ij} = \*$ if $j^{\rm th}$ output is fed to $i^{\rm th}$ input. We define, $[K]:=\{K: K_{ij} = 0, \mbox{~if~} \bK_{ij} = 0\}$. 

\begin{defn}\label{def:struc}
The structured system $(\bA, \bB, \bC)$ and the feedback matrix $\bK$ is said not to have structurally fixed modes (SFMs) if there exists a numerical realization $(A, B, C)$ of $(\bA, \bB, \bC)$
 such that $\cap_{K \in [K]} \sigma(A + BKC) = \phi$, where $\sigma(T)$ denotes the set of eigenvalues of any square matrix $T$. 
 \end{defn}
 
Given a structured system $(\bA, \bB, \bC)$, our aim is to find a minimum set of feedback edges (i.e., sparsest $\bK$) such that the closed-loop structured system $(\bA,\bB,\bC,\bK)$ has no SFMs. Let $\K_s := \{\bK \in \{ 0, \*\}^{m \times p}:(\bA,\bB,\bC,\bK) \mbox{ has no SFMs} \}$. For a structured system $(\bA, \bB, \bC)$, without loss of generality, we assume that $\K_s$ is non-empty. Specifically, $\bK^f \in \K_s$, where $\bK^f_{ij} = \*$ for all $i,j$. Next we describe the problem addressed in this paper.
\begin{prob}\label{prob:one}
Given a structured system $(\bA, \bB, \bC)$, find
\[ \bK^\* ~\in~ \arg\min_{\10n \bK \in \K_s} \norm[\bK]_0. \] 
\end{prob}
Here $\norm[\cdot]_0$ denotes the zero matrix 
   norm\footnote{Although $\norm[\cdot]_0$ does not satisfy all 
   the norm axioms, the number of non-zero entries in a matrix is 
   conventionally referred to as the zero norm.}.
We refer to Problem~\ref{prob:one} as {\it sparsest feedback selection problem}. 

\subsection{Preliminaries}\label{sec:prelim}
Graph theory is a key tool in the analysis of structured systems since a structured system can be represented as a digraph and there exists necessary and sufficient graph theoretic conditions for various structural properties of the system \cite{Rei:88}. Given a structured system $(\bA, \bB, \bC)$ we first construct the system digraph denoted as $\D(\bA, \bB, \bC)$ which is constructed as follows: we define the state digraph $\D(\bA) := \D(V_X, E_X)$ where $V_X = \{x_1,\ldots,x_n\}$ and an edge $(x_j, x_i) \in E_{\x}$ if $\bA_{ij} \neq 0$. Thus a directed edge $(x_j, x_i)$ exists if state $x_j$ can influence state $x_i$. Define the system digraph $\D(\bA, \bB, \bC) := \D(V_{\x}\cup V_{\u}\cup V_{Y}, E_{\x}\cup E_{\u}\cup E_{Y})$, where  $V_{U} = \{u_1, \ldots, u_m \}$ and $V_{Y} = \{y_1, \ldots, y_p \}$. An edge $(u_j, x_i) \in E_{U}$ if $\bB_{ij} \neq 0$ and an edge $(x_j, y_i) \in E_{Y}$ if $\bC_{ij} \neq 0$. Thus a directed edge $(u_j, x_i)$ exists if input $u_j$ can {\it actuate} state $x_i$ and a directed edge $(x_j, y_i)$ exists if output $y_i$ can {\it sense} state $x_j$ and this completes the construction of the system digraph. 

Given a structured system $(\bA, \bB, \bC)$ and a feedback matrix $\bK$, we define the closed-loop system digraph $\D(\bA, \bB, \bC, \bK) := \D(V_{\x}\cup V_{\u}\cup V_{Y}, E_{\x}\cup E_{\u}\cup E_{Y}\cup E_{K})$, where $(y_j, u_i) \in E_K$ if $ \bK_{ij}\neq 0$. Here a directed edge $(y_j, u_i)$ exists if output $y_j$ can be {\it fed to} input $u_i$. 

A digraph is said to be strongly connected if for each ordered pair of vertices $(v_1,v_k)$
there exists an elementary path from $v_1$ to $v_k$. A strongly connected component (SCC) is a subgraph that consists of a maximal set of strongly connected vertices. 
Now, using the closed-loop system digraph $\D(\bA, \bB, \bC, \bK)$ the following result has been shown \cite{PicSezSil:84}.
\begin{prop} [\cite{PicSezSil:84}, Theorem 4]\label{prop:SFM} 
A structured system $(\bA, \bB, \bC)$ have no structurally fixed modes with respect to an information pattern $\bK$ if and only if the following conditions hold:

\noindent a)~in the digraph $\D(\bA, \bB, \bC, \bK)$, each state node $x_i$ is contained in an SCC which includes an edge from $E_K$, and 

\noindent b)~there exists a finite node disjoint union of cycles $\C_g = (V_g, E_g)$ in $\D(\bA, \bB, \bC, \bK)$ where $g$ 
belongs to the set of natural numbers such that $V_X \subseteq \cup_{g}V_g$.
\end{prop}
 
Given a closed-loop structured system $(\bA, \bB, \bC, \bK)$ we can check condition~a) in $O(n^2)$ computations and condition~b) in $O(n^{2.5})$ computations \cite{PapTsi:84}. Thus checking SFMs in a structured system has complexity $O(n^{2.5})$. The objective here is to find a sparsest feedback matrix such that the resulting closed-loop system has no SFMs. We consider structurally cyclic systems with dedicated inputs and outputs. A structurally cyclic system is defined as follows: see \cite{Van:91}.

\begin{defn}\label{def:cyclic}
A structured system $\bA$ is said to be structurally cyclic if the state bipartite graph $\B(\bA)$ has a perfect matching.
\end{defn}
Thus in a structurally cyclic system all state vertices lie in disjoint union of cycles which consists of only $x_i$'s and thus condition~b) in Proposition~\ref{prop:SFM} is satisfied. Thus the feedback selection problem needs to satisfy only condition~a) in Proposition~\ref{prop:SFM}. Henceforth we consider structurally cyclic systems with dedicated inputs and outputs. Thus the following assumption holds.

\begin{assume}\label{asm:cyclic}
The structured system $(\bA, \bB, \bC)$ satisfies $\bB = \mI_n$, $\bC = \mI_n$ and $\B(\bA)$ has a perfect matching.
\end{assume}

The authors in \cite{CarPeqAguKarPap:15} claim that Problem~\ref{prob:one} is NP-hard. 
However, the proof provided is not complete. 
The instance of Problem ~\ref{prob:one} constructed in the NP-hardness proof given in \cite{CarPeqAguKarPap:15} satisfies Assumption~\ref{asm:cyclic}. We show that for this case, the problem can be solved in linear time complexity. 
We mention briefly the error in the proof and the result given in \cite{CarPeqAguKarPap:15} in detail in the next section. 
We also give our algorithm of linear complexity for solving Problem~\ref{prob:one} when Assumption~\ref{asm:cyclic} holds.

\section{Graph Decomposition Problem and the Sparsest Feedback Selection Problem}\label{sec:fallacies}

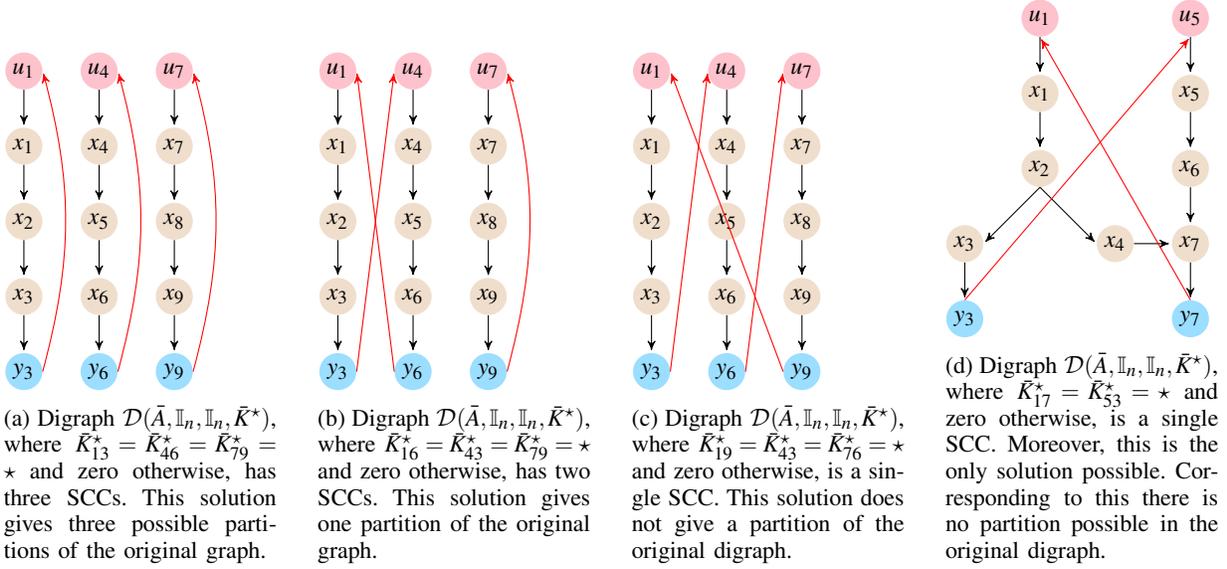
\begin{figure*}[!t]
\begin{center}
\begin{subfigure}[b]{0.2\textwidth}
\begin{tikzpicture}[->,>=stealth',shorten >=1pt,auto,node distance=1.25cm, main node/.style={circle,draw,font=\scriptsize\bfseries}]
\definecolor{myblue}{RGB}{80,80,160}
\definecolor{almond}{rgb}{0.94, 0.87, 0.8}
\definecolor{bubblegum}{rgb}{0.99, 0.76, 0.8}
\definecolor{columbiablue}{rgb}{0.61, 0.87, 1.0}

  \fill[bubblegum] (0, 3) circle (7.0 pt);
  \fill[bubblegum] (1,3) circle (7.0 pt);
  \fill[bubblegum] (2,3) circle (7.0 pt);

  \fill[columbiablue] (0,-1) circle (7.0 pt);
  \fill[columbiablue] (1,-1.0) circle (7.0 pt);
  \fill[columbiablue] (2,-1) circle (7.0 pt);
 
  \fill[almond] (0,0) circle (7.0 pt);
  \fill[almond] (1,0) circle (7.0 pt);
  \fill[almond] (2,0) circle (7.0 pt);
  \fill[almond] (0,1) circle (7.0 pt);
  \fill[almond] (1,1) circle (7.0 pt);
  \fill[almond] (2,1) circle (7.0 pt);
  \fill[almond] (0,2) circle (7.0 pt);
  \fill[almond] (1,2) circle (7.0 pt);
  \fill[almond] (2,2) circle (7.0 pt);

   \node at (0,3) {\small $u_1$};
   \node at (1,3) {\small $u_4$};
   \node at (2,3) {\small $u_7$};

   \node at (0,-1) {\small $y_3$};
   \node at (1,-1) {\small $y_6$};   
   \node at (2,-1) {\small $y_9$};
   
  \node at (0,2) {\small $x_1$};
  \node at (0,1) {\small $x_2$};
  \node at (0,0) {\small $x_3$};
  \node at (1,2) {\small $x_4$}; 
  \node at (1,1) {\small $x_5$};  
  \node at (1,0) {\small $x_6$};   
  \node at (2,2) {\small $x_7$};   
  \node at (2,1) {\small $x_8$};
  \node at (2,0) {\small $x_9$};
  
  \draw (0,1.75)  ->   (0,1.25);
  \draw (0,0.75)  ->   (0,0.25);  
  \draw (1,1.75)  ->   (1,1.25);
  \draw (1,0.75)  ->   (1,0.25);
  \draw (2,1.75)  ->   (2,1.25);
  \draw (2,0.75)  ->   (2,0.25);

  \draw (0,2.75)  ->   (0,2.25);
  \draw (1,2.75)  ->   (1,2.25);
  \draw (2,2.75)  ->   (2,2.25);
  
  \draw (0,-0.25)  ->   (0,-0.75);
  \draw (1,-0.25)  ->   (1,-0.75);
  \draw (2,-0.25)  ->   (2,-0.75);

\path[every node/.style={font=\sffamily\small}]
(0.25,-1) edge[red,bend right = 15] node [left] {} (0.25,3)
(1.25,-1) edge[red,bend right = 15] node [left] {} (1.25,3)
(2.25,-1) edge[red,bend right = 15] node [left] {} (2.25,3);

\end{tikzpicture}
\caption{Digraph $\D(\bA, \mI_n, \mI_n, \bK^\*)$, where $\bK^\*_{13} = \bK^\*_{46} = \bK^\*_{79} = \*$ and zero otherwise, has three SCCs. This solution gives three possible partitions of the original graph.}
\label{fig:counter_eg1}
\end{subfigure}~\hspace{3 mm}
\begin{subfigure}[b]{0.2\textwidth}
\begin{tikzpicture}[->,>=stealth',shorten >=1pt,auto,node distance=1.25cm, main node/.style={circle,draw,font=\scriptsize\bfseries}]
\definecolor{myblue}{RGB}{80,80,160}
\definecolor{almond}{rgb}{0.94, 0.87, 0.8}
\definecolor{bubblegum}{rgb}{0.99, 0.76, 0.8}
\definecolor{columbiablue}{rgb}{0.61, 0.87, 1.0}

  \fill[bubblegum] (0, 3) circle (7.0 pt);
  \fill[bubblegum] (1,3) circle (7.0 pt);
  \fill[bubblegum] (2,3) circle (7.0 pt);

  \fill[columbiablue] (0,-1) circle (7.0 pt);
  \fill[columbiablue] (1,-1.0) circle (7.0 pt);
  \fill[columbiablue] (2,-1) circle (7.0 pt);
 
  \fill[almond] (0,0) circle (7.0 pt);
  \fill[almond] (1,0) circle (7.0 pt);
  \fill[almond] (2,0) circle (7.0 pt);
  \fill[almond] (0,1) circle (7.0 pt);
  \fill[almond] (1,1) circle (7.0 pt);
  \fill[almond] (2,1) circle (7.0 pt);
  \fill[almond] (0,2) circle (7.0 pt);
  \fill[almond] (1,2) circle (7.0 pt);
  \fill[almond] (2,2) circle (7.0 pt);

   \node at (0,3) {\small $u_1$};
   \node at (1,3) {\small $u_4$};
   \node at (2,3) {\small $u_7$};

   \node at (0,-1) {\small $y_3$};
   \node at (1,-1) {\small $y_6$};   
   \node at (2,-1) {\small $y_9$};
   
  \node at (0,2) {\small $x_1$};
  \node at (0,1) {\small $x_2$};
  \node at (0,0) {\small $x_3$};
  \node at (1,2) {\small $x_4$}; 
  \node at (1,1) {\small $x_5$};  
  \node at (1,0) {\small $x_6$};   
  \node at (2,2) {\small $x_7$};   
  \node at (2,1) {\small $x_8$};
  \node at (2,0) {\small $x_9$};
  
  \draw (0,1.75)  ->   (0,1.25);
  \draw (0,0.75)  ->   (0,0.25);  
  \draw (1,1.75)  ->   (1,1.25);
  \draw (1,0.75)  ->   (1,0.25);
  \draw (2,1.75)  ->   (2,1.25);
  \draw (2,0.75)  ->   (2,0.25);

  \draw (0,2.75)  ->   (0,2.25);
  \draw (1,2.75)  ->   (1,2.25);
  \draw (2,2.75)  ->   (2,2.25);
  
  \draw (0,-0.25)  ->   (0,-0.75);
  \draw (1,-0.25)  ->   (1,-0.75);
  \draw (2,-0.25)  ->   (2,-0.75);
  
  \draw [red] (0.25,-1)  ->   (0.75,3);
  \draw [red] (0.75,-1)  ->   (0.25,3);

\path[every node/.style={font=\sffamily\small}]
(2.25,-1) edge[red,bend right = 15] node [left] {} (2.25,3);

\end{tikzpicture}
\caption{Digraph $\D(\bA, \mI_n, \mI_n, \bK^\*)$, where $\bK^\*_{16} = \bK^\*_{43} = \bK^\*_{79} = \*$ and zero otherwise, has two SCCs. This solution gives one partition of the original graph.}
\label{fig:counter_eg2}
\end{subfigure}~\hspace{3 mm}
\begin{subfigure}[b]{0.2\textwidth}
\begin{tikzpicture}[->,>=stealth',shorten >=1pt,auto,node distance=1.25cm, main node/.style={circle,draw,font=\scriptsize\bfseries}]
\definecolor{myblue}{RGB}{80,80,160}
\definecolor{almond}{rgb}{0.94, 0.87, 0.8}
\definecolor{bubblegum}{rgb}{0.99, 0.76, 0.8}
\definecolor{columbiablue}{rgb}{0.61, 0.87, 1.0}

  \fill[bubblegum] (0, 3) circle (7.0 pt);
  \fill[bubblegum] (1,3) circle (7.0 pt);
  \fill[bubblegum] (2,3) circle (7.0 pt);

  \fill[columbiablue] (0,-1) circle (7.0 pt);
  \fill[columbiablue] (1,-1.0) circle (7.0 pt);
  \fill[columbiablue] (2,-1) circle (7.0 pt);
 
  \fill[almond] (0,0) circle (7.0 pt);
  \fill[almond] (1,0) circle (7.0 pt);
  \fill[almond] (2,0) circle (7.0 pt);
  \fill[almond] (0,1) circle (7.0 pt);
  \fill[almond] (1,1) circle (7.0 pt);
  \fill[almond] (2,1) circle (7.0 pt);
  \fill[almond] (0,2) circle (7.0 pt);
  \fill[almond] (1,2) circle (7.0 pt);
  \fill[almond] (2,2) circle (7.0 pt);

   \node at (0,3) {\small $u_1$};
   \node at (1,3) {\small $u_4$};
   \node at (2,3) {\small $u_7$};

   \node at (0,-1) {\small $y_3$};
   \node at (1,-1) {\small $y_6$};   
   \node at (2,-1) {\small $y_9$};
   
  \node at (0,2) {\small $x_1$};
  \node at (0,1) {\small $x_2$};
  \node at (0,0) {\small $x_3$};
  \node at (1,2) {\small $x_4$}; 
  \node at (1,1) {\small $x_5$};  
  \node at (1,0) {\small $x_6$};   
  \node at (2,2) {\small $x_7$};   
  \node at (2,1) {\small $x_8$};
  \node at (2,0) {\small $x_9$};
  
  \draw (0,1.75)  ->   (0,1.25);
  \draw (0,0.75)  ->   (0,0.25);  
  \draw (1,1.75)  ->   (1,1.25);
  \draw (1,0.75)  ->   (1,0.25);
  \draw (2,1.75)  ->   (2,1.25);
  \draw (2,0.75)  ->   (2,0.25);

  \draw (0,2.75)  ->   (0,2.25);
  \draw (1,2.75)  ->   (1,2.25);
  \draw (2,2.75)  ->   (2,2.25);
  
  \draw (0,-0.25)  ->   (0,-0.75);
  \draw (1,-0.25)  ->   (1,-0.75);
  \draw (2,-0.25)  ->   (2,-0.75);
  
  \draw [red] (0.25,-1)  ->   (0.75,3);
  \draw [red] (1.25,-1)  ->   (1.75,3);
  \draw [red] (1.75,-1)  ->   (0.25,3);

\end{tikzpicture}
\caption{Digraph $\D(\bA, \mI_n, \mI_n, \bK^\*)$, where $\bK^\*_{19} = \bK^\*_{43} = \bK^\*_{76} = \*$ and zero otherwise, is a single SCC. This solution does not give a partition of the original digraph. }
\label{fig:counter_eg3}
\end{subfigure}~\hspace{3 mm}
\begin{subfigure}[b]{0.2\textwidth}
\begin{tikzpicture}[->,>=stealth',shorten >=1pt,auto,node distance=1.25cm, main node/.style={circle,draw,font=\scriptsize\bfseries}]
\definecolor{myblue}{RGB}{80,80,160}
\definecolor{almond}{rgb}{0.94, 0.87, 0.8}
\definecolor{bubblegum}{rgb}{0.99, 0.76, 0.8}
\definecolor{columbiablue}{rgb}{0.61, 0.87, 1.0}

  \fill[bubblegum] (0, 3) circle (7.0 pt);
  \fill[bubblegum] (2,3) circle (7.0 pt);

  \fill[columbiablue] (-1,-1) circle (7.0 pt);     
   \fill[columbiablue] (2,-1) circle (7.0 pt);
 
  \fill[almond] (0,2) circle (7.0 pt);
  \fill[almond] (0,1) circle (7.0 pt);
  \fill[almond] (-1,0) circle (7.0 pt);
  \fill[almond] (1,0) circle (7.0 pt);  
  \fill[almond] (2,2) circle (7.0 pt);
  \fill[almond] (2,1) circle (7.0 pt);  
  \fill[almond] (2,0) circle (7.0 pt);        

   \node at (0,3) {\small $u_1$};
   \node at (2,3) {\small $u_5$};

   \node at (-1,-1) {\small $y_3$};
   \node at (2,-1) {\small $y_7$};   
   
  \node at (0,2) {\small $x_1$};
  \node at (0,1) {\small $x_2$};
  \node at (-1,0) {\small $x_3$};  
  \node at (1,0) {\small $x_4$};   
  \node at (2,2) {\small $x_5$}; 
  \node at (2,1) {\small $x_6$};  
  \node at (2,0) {\small $x_7$};
  
  \draw (0,1.75)  ->   (0,1.25);
  \draw (0,0.75)  ->   (-0.75,0);
  \draw (0,0.75)  ->   (0.75,0);   
  \draw (2,1.75)  ->   (2,1.25);
  \draw (2,0.75)  ->   (2,0.25);
  \draw (1.25,0)  ->   (1.75,0);
  
  \draw (2,-0.25)  ->   (2,-0.75);
  \draw (-1,-0.25)  ->   (-1,-0.75);
  
  \draw (2,2.75)  ->   (2,2.25);
  \draw (0,2.75)  ->   (0,2.25);
  
  \draw [red] (2,-0.75)  ->   (0,2.75);
  \draw [red] (-1,-0.75)  ->   (2,2.75);
  
\end{tikzpicture}
\caption{Digraph $\D(\bA, \mI_n, \mI_n, \bK^\*)$, where $\bK^\*_{17} = \bK^\*_{53} = \*$ and zero otherwise, is a single SCC. Moreover, this is the only solution possible. Corresponding to this there is no partition possible in the original digraph. }
\label{fig:counter_eg4}
\end{subfigure}~\hspace{3 mm}
\caption{Illustrative figure demonstrating the incompleteness in the reduction given in \cite{CarPeqAguKarPap:15}. In all the figures, 
	each state vertex $x_k$ also has a self loop which we have omitted for the sake of clarity. Also, each state vertex $x_k$ has input $u_k$ and output $y_k$ connected which are omitted for many $x_k$'s for the sake of clarity.}
\label{fig:counter_eg}
\end{center}
\end{figure*}

In this section we focus on the subtle difference between the approach followed
in \cite{CarPeqAguKarPap:15} for solving the sparsest feedback selection problem, in
which the authors link this problem with the graph decomposition problem. In this
section we elaborate on a key error in their reduction procedure.
We later consider a counter-example that helps understand the error.

Recall that Problem~\ref{prob:one} aims at finding a sparsest feedback matrix $\bK$ 
such that the closed-loop system has no SFMs. 
This problem has been shown to be NP-hard in \cite{CarPeqAguKarPap:15} using 
reduction from a known NP-complete problem, the graph decomposition problem \cite{Tar:84},
described below.
\begin{prob}[Graph Decomposition Problem \cite{Tar:84}]
	\label{prob:graph_partition}
	Given a directed acyclic graph $\D(V,E)$, find a partition of $V$ into two non-empty sets $\Gamma_1$ and $\Gamma_2$ such that
	
	\noindent
	$\bullet$ no edge in $E$ connects any vertex in $\Gamma_1$ to a vertex in $\Gamma_2$;
	
	\noindent
	$\bullet$ for every vertex $v \in \Gamma_i$, $i= 1,2$, there exists a source-sink pair and a path between them such that the path contains $v$ and passes only through the nodes in $\Gamma_i$. Here, source (sink, resp.) refers to a node that does not have any incoming (outgoing, resp.) edge.  
\end{prob}


The authors of  \cite{CarPeqAguKarPap:15} have taken a general instance of 
Problem~\ref{prob:graph_partition} and constructed an instance of Problem~\ref{prob:one}. 
Next, in order to claim NP-hardness of Problem~\ref{prob:one}, it 
is shown that if an optimal solution $\bK^\*$ to Problem~\ref{prob:one} results 
in {\em two SCCs} in $\D(\bA,\bB,\bC,\bK^\*)$, then the original graph 
can be decomposed as required.
However, this result is not enough. 
For proving the NP-hardness of a given problem using reduction, one must reduce 
an arbitrary instance of a known NP-complete or NP-hard problem to an 
instance of the given problem with polynomial complexity. 
Further, the reduction must also be such that {\it \underline{any}} optimal solution to 
the given problem must give an optimal solution to the NP-complete or 
NP-hard problem chosen and vice versa \cite{GarJoh:02}.

However, \cite{CarPeqAguKarPap:15} only shows that an optimal solution $\bK^\*$ to Problem~\ref{prob:one} that has two SCCs in $\D(\bA, \bB, \bC, \bK^\*)$ gives a solution to the decomposition problem. 
This does not answer the case where an optimal solution $\bK^\*$ results in a 
different number of SCCs in $\D(\bA, \bB, \bC, \bK^\*)$. 
Importantly, one must also address the case when an optimal solution results 
in a {\em single} SCC and specify what happens in the graph decomposition problem then. 
More precisely, for the reduction to be complete, one should show that any optimal solution to Problem~\ref{prob:one} gives an optimal solution to the graph decomposition problem and vice-versa. 
Thus, if an optimal solution $\bK^\*$ to Problem~\ref{prob:one} exists that results in a single SCC in $\D(\bA, \bB, \bC, \bK^\*)$, then it is not clear about the corresponding solution of the graph decomposition problem. 

We demonstrate the ambiguity in the proof using the illustrative examples given in Figure~\ref{fig:counter_eg}. Note that in Figure~\ref{fig:counter_eg1}, the optimal solution $\bK^\*$ results in three SCCs. However, there are three possible partitioning of the original graph that can be obtained from this. These partitions are: ${\Gamma_{1}^1} = \{x_1, \ldots, x_6\}$ and $\Gamma_2^1 = \{x_7, x_8, x_9\}$, $\Gamma_1^2 = \{x_1, x_2, x_3\}$ and $\Gamma_2^2 = \{x_4, \ldots, x_9\}$ and $\Gamma_1^3 = \{x_1, x_2, x_3, x_7, x_8, x_9\}$ and $\Gamma_2^3 = \{x_4, x_5, x_6\}$. 
Now another optimal solution given in Figure~\ref{fig:counter_eg2} results in two SCCs. 
Corresponding to this solution, there is a partitioning of the original graph, $\Gamma_1 = \{x_1, \ldots, x_6\}$ and $\Gamma_2 = \{x_7, x_8, x_9\}$. 
The optimal solution given in Figure~\ref{fig:counter_eg3} results in a single SCC. 
In this case, according to \cite{CarPeqAguKarPap:15}, from $\bK^\*$ it is not possible to say whether there exists a partitioning of the original problem. Note that in this example, many optimal $\bK^\*$ are possible, and the original graph can be decomposed irrespective of the solution
chosen. 
On the contrary, consider the system given in Figure~\ref{fig:counter_eg4} and the corresponding optimal solution. Notice that for this structured system, given $\bK^\*$ is the only optimal solution to Problem~\ref{prob:one} and it results in a single SCC. Moreover, the original graph can not be decomposed. Thus, the case where optimal solution results in a single SCC, we can not conclude either way for the graph decomposition problem.
Thus, the reduction in \cite{CarPeqAguKarPap:15} is inconclusive.

In summary, in our opinion, the problem that is shown to be 
NP-hard in \cite{CarPeqAguKarPap:15} can be stated as:
given a structured system $(\bA,\bB,\bC)$, find $\bK$ such 
that $\D(\bA, \bB, \bC, \bK)$ has two or more SCCs.
Note that this is slightly different than Problem~\ref{prob:one}. 
In the next section, we provide a linear time algorithm for solving Problem~\ref{prob:one}.

\section{Linear Complexity Algorithm for Solving Problem~\ref{prob:one}}\label{sec:poly}
The structured system $(\bA, \bB, \bC)$ considered in the NP-hard proof given in \cite{CarPeqAguKarPap:15} satisfies Assumption~\ref{asm:cyclic}. In this section we show that if the structured system satisfies Assumption~\ref{asm:cyclic}, then Problem~\ref{prob:one} can be solved in linear time. The proposed solution is the consequence of the following important observation. 
\begin{lem}\label{lem:single_SCC}
Consider a structured system $(\bA, \bB, \bC)$. Let $\bK^\*$ be an optimal solution to Problem~\ref{prob:one} such that all state nodes lie in $\beta$ number of SCCs in $\D(\bA, \bB, \bC, \bK^\*)$, where $\beta > 1$. Then, there exists another optimal solution $\bK^\*_{\rm new}$ such that $\norm[\bK^\*]_0 = \norm[\bK^\*_{\rm new}]_0$ and all state nodes lie in $\beta - 1$ number of SCCs in $\D(\bA, \bB, \bC, \bK^\*_{\rm new})$.
\end{lem}
\begin{proof}
Given $\bK^\*$ is an optimal solution to Problem~\ref{prob:one} and  all state nodes lie in $\beta$ number of SCCs in $\D(\bA, \bB, \bC, \bK^\*)$, say $\cCl_1, \ldots, \cCl_\beta$. Pick two SCCs, say $\cCl_i, \cCl_j$. Since $\bK^\*$ satisfies condition~a) both $\cCl_i$ and $\cCl_j$ has a feedback edge in it. Let $(y_a, u_b) \in \cCl_i$ and $(y_c, u_d) \in \cCl_j$. 
Now, break the edges $(y_a, u_b), (y_c, u_d)$ and make the edges $(y_a, u_d), (y_c, u_b)$. We claim that now all the nodes in $\cCl_i$ and $\cCl_j$ lie in a single SCC with feedback edges $(y_a, u_d), (y_c, u_b)$. To prove this we need to show that there exists a directed path between two arbitrary vertices in them. Consider any four arbitrary vertices $v_p, v_q \in \cCl_i$ and $v_\zeta, v_\delta \in \cCl_j$. Since $\cCl_i$ is an SCC, notice that there exists a directed path from $u_b$ to $v_p$. Similarly, there exists a directed path from $v_p$ to $y_a$ also. Thus there exists a directed path from $u_b$ to $y_a$ passing through $v_p$. Similarly, we can show a directed path from  $u_b$ to $y_a$ passing through $v_q$. Further, using the same argument on $\cCl_j$ we can show that there exists a directed path from $u_d$ to $y_c$ passing through $v_\zeta$ and from $u_d$ to $y_c$ passing through $v_\delta$. These paths are shown using dotted lines in the Figure~\ref{fig:single_SCC}. 
Now, on adding edges $(y_a, u_d), (y_c, u_b)$ all these vertices $v_p, v_q, v_\zeta$ and $v_\delta$ lie in cycles as shown in the Figure~\ref{fig:single_SCC}. Thus there exists a path between any two arbitrary vertices in $\cCl_i$, any two arbitrary vertices in $\cCl_j$ and any two arbitrary vertices in $\cCl_i, \cCl_j$. Thus 
by breaking edges $(y_a, u_b), (y_c, u_d)$ and making edges $(y_a, u_d), (y_c, u_b)$ all the vertices in $\cCl_i$ and $\cCl_j$ lie in a single SCC. Thus given an optimal feedback matrix, there exists another feedback matrix with the same number of edges, hence optimal, such that all the state nodes are spanned by one less number of SCCs in the closed-loop system digraph. This completes the proof.
\end{proof}

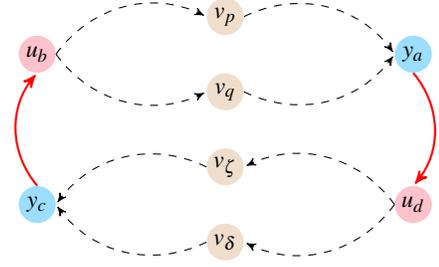
\begin{figure}[H]
\begin{center}
\begin{tikzpicture}[->,>=stealth',shorten >=1pt,auto,node distance=1.25cm, main node/.style={circle,draw,font=\scriptsize\bfseries}]
\definecolor{myblue}{RGB}{80,80,160}
\definecolor{almond}{rgb}{0.94, 0.87, 0.8}
\definecolor{bubblegum}{rgb}{0.99, 0.76, 0.8}
\definecolor{columbiablue}{rgb}{0.61, 0.87, 1.0}

  \fill[bubblegum] (0, 0) circle (7.0 pt);
  \fill[almond] (2.5,0.5) circle (7.0 pt);
  \fill[almond] (2.5,-0.5) circle (7.0 pt);
  \fill[columbiablue] (5,0) circle (7.0 pt);

   \node at (5,0) {\small $y_a$};
   \node at (2.5,0.5) {\small $v_p$};
   \node at (2.5,-0.5) {\small $v_q$};  
   \node at (0,0) {\small $u_b$};

  \fill[columbiablue] (0, -2) circle (7.0 pt);
  \fill[almond] (2.5,-1.5) circle (7.0 pt);
  \fill[almond] (2.5,-2.5) circle (7.0 pt);
  \fill[bubblegum] (5,-2) circle (7.0 pt);

   \node at (5,-2) {\small $u_d$};
   \node at (2.5,-1.5) {\small $v_\zeta$};
   \node at (2.5,-2.5) {\small $v_\delta$};  
   \node at (0,-2) {\small $y_c$};  

\path[every node/.style={font=\sffamily\small}]
(0.25,0) edge[dashed,bend left = 40] node [left] {} (2.25,0.5)
(2.75,0.5) edge[dashed, bend left = 40] node [left] {} (4.75,0)
(0.25,0) edge[dashed,bend right = 40] node [left] {} (2.25,-0.5)
(2.75,-0.5) edge[dashed, bend right = 40] node [left] {} (4.75,0)

(2.25,-1.5) edge[dashed,bend right = 40] node [left] {} (0.25,-2) 
(4.75,-2) edge[dashed, bend right = 40] node [left] {} (2.75,-1.5)
(2.25,-2.5) edge[dashed,bend left = 40] node [left] {} (0.25,-2)
(4.75,-2) edge[dashed, bend left = 40] node [left] {} (2.75,-2.5)

(5,-0.25) edge[red,thick,bend left = 40] node [left] {} (5,-1.75)
(0,-1.75) edge[red,thick, bend left = 40] node [left] {} (0,-0.25);
\end{tikzpicture}
\caption{Schematic diagram depicting the construction used in the proof of Lemma~\ref{lem:single_SCC}. Dotted lines between two vertices here denotes existence of a directed path between them.}
\label{fig:single_SCC}
\end{center}
\end{figure}
As a consequence of Lemma~\ref{lem:single_SCC}, we have the following corollary.

\begin{cor}\label{cor:SCC}
Consider a structured system $(\bA, \bB, \bC)$. Then, there exists an optimal solution $\bK^\*$ to Problem~\ref{prob:one} such that all state nodes lie in a single SCC in $\D(\bA, \bB, \bC, \bK^\*)$.
\end{cor}
The above corollary is true, since given any optimal solution $\bK'$ to Problem~\ref{prob:one}, we can apply Lemma~\ref{lem:single_SCC} recursively such that we arrive at an optimal solution $\bK^\*$ such that all state nodes lie in a single SCC in $\D(\bA, \bB, \bC, \bK^\*)$. 
Thus solving Problem~\ref{prob:one} on a structured system $(\bA, \bB, \bC)$ is same as finding the minimum number of feedback edges to add in the digraph $\D(\bA, \bB, \bC)$ such that in the resulting digraph all state nodes lie in an SCC. 

 The problem of finding minimum number of edges to add in a digraph that the resulting graph is strongly connected is referred as {\it strong connectivity augmentation problem} \cite{EswTar:76}. We briefly explain the strong connectivity augmentation problem here for the sake of completeness. Given a directed graph $\D(V, E)$, the strong connectivity augmentation problem aims at finding the minimum cardinality set of edges $E'$ such that $\D(V, E \cup E')$ is strongly connected. First note that if $\D(\bA)$ is irreducible, then $E' = \phi$ and any $\bK$ that has
a single non-zero entry is optimal. Hence, from now on we only focus on the non-trivial cases such that $\D(\bA)$ has at least two SCCs. There exists a linear time algorithm for solving the above problem optimally \cite {Rag:05}. Given a directed graph the algorithm given in \cite {Rag:05} gives the minimum set of edges that when added to the graph results in a single SCC. Using this result now we give the linear time algorithm to solve Problem~\ref{prob:one} on structured systems that satisfy Assumption~\ref{asm:cyclic}:

\noindent Step~1: Given a structured system $(\bA, \bB = \mI_n, \bC = \mI_n)$, solve the strong connectivity augmentation problem on the digraph $\D(\bA)$. Let $E_\X$ denotes the optimal solution obtained. 

\noindent Step~2: Define $\bK = \{ \bK_{ij} = \*: (x_j, x_i) \in E_\X \}$.

Note that connecting defining $\bK$ as given in Step~2 is possible since $\bB = \bC = \mI_n$. Now we prove that $\bK$ obtained in Step~2 is an optimal solution to Problem~\ref{prob:one}.

\begin{theorem}\label{th:linear}
Consider a structured system $(\bA, \bB, \bC)$ such that Assumption~\ref{asm:cyclic} holds. Then, Problem~\ref{prob:one} can be solved in $O(n)$ complexity, where $n$ denotes the number of states.
\end{theorem}
\begin{proof}
Here we prove that solving strong connectivity augmentation problem on $\D(\bA)$ gives an optimal solution to Problem~\ref{prob:one}. The structured system given has to satisfy one of the following cases: i)~$\D(\bA)$ is irreducible; ii)~$\D(\bA)$ is reducible. In case~i), solution to the strong connectivity augmentation problem, $E_\X = \phi$. Then, an optimal solution to Problem~\ref{prob:one} is given by $\{\bK: \bK_{11} = \* \mbox{~and~} 0 \mbox{~otherwise~}\}$. In case~ii), we prove that $\bK$ obtained in Step~2 corresponding to $E_\X$ is an optimal solution to Problem~\ref{prob:one}. We first show that $\bK$ is a feasible solution, i.e., $\bK \in \K_s$. By the construction of $\bK$ given in Step~2 notice that all state nodes lie in a single SCC in $\D(\bA, \bB, \bC, \bK)$. Also, $\D(\bA)$ is not irreducible. Thus condition~a) is satisfied for all states. Thus $\bK \in \K_s$. Now we prove that $\bK$ is an optimal solution to Problem~\ref{prob:one}, i.e., $\norm[\bK]_0 = \norm[\bK^\*]_0$. Suppose not. Then there exists $\bK' \in \K_s$ such that $\norm[\bK']_0 < \norm[\bK]_0$ and by Corollary~\ref{cor:SCC} all state nodes lie in a single SCC in $\D(\bA, \bB, \bC, \bK')$. Consider edges $E'_\X$ where $(x_j, x_i) \in E'_\X$ if $\bK'_{ij} = \*$. Notice that $|E'_\X| < |E_\X|$ and $\D(V_X, E_X \cup E'_\X)$ is an SCC. This contradicts the assumption that $E_\X$ is an optimal solution to the strong connectivity augmentation problem. This proves that the feedback matrix obtained by solving the strong connectivity augmentation problem on $\D(\bA)$, $\bK = \{ \bK_{ij} = \*: (x_j, x_i) \in E_\X \}$, is an optimal solution to Problem~\ref{prob:one}.

Now the complexity of the strong connectivity augmentation algorithm is linear in the number of nodes in the digraph. Since $|V_X| = n$, the result follows.  \\
\end{proof}
Though we show that when $\bB = \bC = \mI_n$ and all feedback links are feasible, this algorithm gives an optimal solution to Problem~\ref{prob:one}, these results do not immediately extend to the cases where some feedback links are not feasible or feedback links are associated with costs. We believe that these problems are NP-hard and approximation algorithms fro these problems will be subject of future work.
%
%
\section{Conclusion}\label{sec:conclu}
This paper deals with optimal feedback selection of structured systems. The objective here is to obtain a sparsest feedback matrix such that the resulting closed-loop system has no structurally fixed modes. This problem was considered in 
Carvalho et.al in \cite{CarPeqAguKarPap:15, CarPeqAguKarPap:15_arx} 
though we elaborated in this paper on an error in their proof of NP-hardness.
We have shown recently in \cite{MooChaBel:17_arxMinCostFeedback} that 
this problem is NP-hard.
Further, in this paper, we also proved that solving this problem is, 
in fact, not NP-hard on 
the subclass of systems considered in \cite{CarPeqAguKarPap:15}, i.e., 
structurally cyclic with dedicated inputs and outputs. 
Finally, we provided an algorithm for this subclass of systems 
that has linear complexity in the number of states of the system.



\end{document}